\documentclass{amsart}%
\usepackage{amsmath}
\usepackage{amssymb}
\usepackage{amsfonts}%
\usepackage{graphicx}
\providecommand{\U}[1]{\protect \rule{.1in}{.1in}}
\newtheorem{theorem}{Theorem}
\theoremstyle{plain}

\newtheorem{corollary}{Corollary}

\newtheorem{definition}{Definition}

\newtheorem{remark}{Remark}

\numberwithin{equation}{section}
\begin{document}
\title[vanishing generalized Morrey estimates]{A class of sublinear operators and their commutators by with rough
kernels{\ }on vanishing generalized Morrey spaces}
\author{FER\.{I}T G\"{U}RB\"{U}Z}
\address{HAKKARI UNIVERSITY, FACULTY OF EDUCATION, DEPARTMENT OF MATHEMATICS EDUCATION,
HAKKARI, TURKEY }
\curraddr{}
\email{feritgurbuz@hakkari.edu.tr}
\urladdr{}
\thanks{}
\thanks{}
\thanks{}
\date{}
\dedicatory{ }
\begin{abstract}
In this paper, we consider the boundedness of a class of sublinear operators
and their commutators by with rough kernels associated with
{Calder\'{o}n--Zygmund operator, Hard-Littlewood maximal operator, fractional
integral operator, }fractional maximal operator by with rough kernels both on
vanishing generalized Morrey spaces and vanishing Morrey spaces, respectively.

\end{abstract}
\maketitle

\section{Introduction and useful informations}

To characterize the regularity of solutions to some partial differential
equations(PDEs), Morrey \cite{Morrey} first introduced classical Morrey spaces
$M_{p,\lambda}$ which naturally are generalizations of Lebesgue spaces.

We will say that a function $f\in M_{p,\lambda}=M_{p,\lambda}\left(
{\mathbb{R}^{n}}\right)  $ if%
\begin{equation}
\sup_{x\in{\mathbb{R}^{n},r>0}}\left[  r^{-\lambda}%
%TCIMACRO{\dint \limits_{B(x,r)}}%
%BeginExpansion
{\displaystyle \int \limits_{B(x,r)}}
%EndExpansion
\left \vert f\left(  y\right)  \right \vert ^{p}dy\right]  ^{1/p}<\infty
.\label{1.1}%
\end{equation}
Here, $1<p<\infty$ and $0<\lambda<n$ and the quantity of (\ref{1.1}) is the
$\left(  p,\lambda \right)  $-Morrey norm, denoted by $\left \Vert f\right \Vert
_{M_{p,\lambda}}$. We also refer to \cite{Adams, Gurbuz0} for the latest
research on the theory of Morrey spaces associated with harmonic analysis. On
the other hand, the study of the operators of harmonic analysis in vanishing
Morrey space, in fact has been almost not touched. A version of the classical
Morrey space $M_{p,\lambda}({\mathbb{R}^{n}})$ where it is possible to
approximate by "nice" functions is the so called vanishing Morrey space
$VM_{p,\lambda}({\mathbb{R}^{n}})$ has been introduced by Vitanza in
\cite{Vitanza1} and has been applied there to obtain a regularity result for
elliptic PDEs. This is a subspace of functions in $M_{p,\lambda}%
({\mathbb{R}^{n}})$, which satisfies the condition%
\[
\lim_{r\rightarrow0}\sup_{\underset{0<t<r}{x\in{\mathbb{R}^{n}}}}\left[
t^{-\lambda}%
%TCIMACRO{\dint \limits_{B(x,t)}}%
%BeginExpansion
{\displaystyle \int \limits_{B(x,t)}}
%EndExpansion
\left \vert f\left(  y\right)  \right \vert ^{p}dy\right]  ^{1/p}=0,
\]
where $1<p<\infty$ and $0<\lambda<n$ for brevity, so that%
\[
VM_{p,\lambda}({\mathbb{R}^{n}})=\left \{  f\in M_{p,\lambda}({\mathbb{R}^{n}%
}):\lim_{r\rightarrow0}\sup_{\underset{0<t<r}{x\in{\mathbb{R}^{n}}}}%
t^{-\frac{\lambda}{p}}\Vert f\Vert_{L_{p}(B(x,t))}=0\right \}  .
\]

Later in \cite{Vitanza2} Vitanza has proved an existence theorem for a
Dirichlet problem, under weaker assumptions than in \cite{Miranda} and a
$W^{3,2}$ regularity result assuming that the partial derivatives of the
coefficients of the highest and lower order terms belong to vanishing Morrey
spaces depending on the dimension. For the properties and applications of
vanishing Morrey spaces, see also \cite{Cao-Chen}. It is known that, there is
no research regarding boundedness of the sublinear operators with rough kernel
on vanishing Morrey spaces.

Let $\Omega \in L_{s}(S^{n-1})$ with $1<s\leq \infty$ be homogeneous of degree
zero and satisfies the cancellation condition
\begin{equation}
\int \limits_{S^{n-1}}\Omega(x^{\prime})d\sigma(x^{\prime})=0,\label{0}%
\end{equation}
where $x^{\prime}=\frac{x}{|x|}$ for any $x\neq0$. We define $s^{\prime}%
=\frac{s}{s-1}$ for any $s>1$. Suppose that $T_{\Omega}$ represents a linear
or a sublinear operator, which satisfies that for any $f\in L_{1}%
({\mathbb{R}^{n}})$ with compact support and $x\notin suppf$
\begin{equation}
|T_{\Omega}f(x)|\leq c_{0}\int \limits_{{\mathbb{R}^{n}}}\frac{|\Omega
(x-y)|}{|x-y|^{n}}\,|f(y)|\,dy,\label{e1}%
\end{equation}
where $c_{0}$ is independent of $f$ and $x$. Similarly, we assume that
$T_{\Omega,\alpha}$, $\alpha \in \left(  0,n\right)  $ represents a linear or a
sublinear operator, which satisfies that for any $f\in L_{1}({\mathbb{R}^{n}%
})$ with compact support and $x\notin suppf$%

\begin{equation}
|T_{\Omega,\alpha}f(x)|\leq c_{0}\int \limits_{{\mathbb{R}^{n}}}\frac
{|\Omega(x-y)|}{|x-y|^{n-\alpha}}\,|f(y)|\,dy,\label{e2}%
\end{equation}
for some $\alpha \in \left(  0,n\right)  $, where $c_{0}$ is independent of $f$
and $x$.

We point out that the condition (\ref{e1}) in the case $\Omega \equiv1$ was
first introduced by Soria and Weiss in \cite{SW} . The conditions (\ref{e1})
and (\ref{e2}) are satisfied by many interesting operators in harmonic
analysis, such as the Calder\'{o}n-Zygmund (C-Z) operators, Carleson's maximal
operator, Hardy-Littlewood (H-L) maximal operator, C. Fefferman's singular
multipliers, R. Fefferman's singular integrals, Ricci-Stein's oscillatory
singular integrals, the Bochner-Riesz means, fractional Marcinkiewicz
operator, fractional maximal operator, fractional integral operator (Riesz
potential) and so on (see \cite{SW} for details).

Let $f\in L_{1}^{loc}\left(  {\mathbb{R}^{n}}\right)  $. The C-Z singular
integral operator $\overline{T}_{\Omega}$ and H-L maximal operator $M_{\Omega
}$ by with rough kernels are defined by%

\[
\overline{T}_{\Omega}f(x)=p.v.\int \limits_{{\mathbb{R}^{n}}}\frac{\Omega
(x-y)}{|x-y|^{n}}f(y)dy,
\]%
\[
M_{\Omega}f\left(  x\right)  =\sup_{t>0}\frac{1}{\left \vert B\left(
x,t\right)  \right \vert }\int \limits_{B\left(  x,t\right)  }\left \vert
\Omega \left(  x-y\right)  \right \vert \left \vert f\left(  y\right)
\right \vert dy
\]
satisfy condition (\ref{e1}), where a homogeneous of degree zero function
$\Omega \left(  x^{\prime}\right)  $ satisfies (\ref{0}) on the unit sphere and
belongs to $\Omega \in L_{s}(S^{n-1})$ with $1<s\leq \infty$.

It is obvious that when $\Omega \equiv1$, $\overline{T}_{\Omega}\equiv
\overline{T}$ and $M_{\Omega}\equiv M$ are the standard C-Z singular intagral
operator, briefly a C-Z operator and the H-L maximal operator, respectively.

On the other hand, in 1971, Muckenhoupt and Wheeden \cite{Muckenhoupt} defined
the fractional integral operator with rough kernel $\overline{T}%
_{\Omega,\alpha}$ by%

\[
\overline{T}_{\Omega,\alpha}f(x)=\int \limits_{{\mathbb{R}^{n}}}\frac
{\Omega(x-y)}{|x-y|^{n-\alpha}}f(y)dy\qquad0<\alpha<n
\]
and a related fractional maximal operator with rough kernel $M_{\Omega,\alpha
}$ is given by%

\[
M_{\Omega,\alpha}f(x)=\sup_{t>0}|B(x,t)|^{-1+\frac{\alpha}{n}}\int
\limits_{B(x,t)}\left \vert \Omega \left(  x-y\right)  \right \vert
|f(y)|dy\qquad0<\alpha<n,
\]
where $\Omega \in L_{s}(S^{n-1})$ with $1<s\leq \infty$ is homogeneous of degree
zero on ${\mathbb{R}^{n}}$ and also $\overline{T}_{\Omega,\alpha}$ and
$M_{\Omega,\alpha}$ satisfy condition (\ref{e2}).

If $\alpha=0$, then $M_{\Omega,0}\equiv M_{\Omega}$ and $\overline{T}%
_{\Omega,0}\equiv \overline{T}_{\Omega}$, respectively. It is obvious that when
$\Omega \equiv1$, $M_{1,\alpha}\equiv M_{\alpha}$ and $\overline{T}_{1,\alpha
}\equiv \overline{T}_{\alpha}$ are the fractional maximal operator and the
fractional integral operator (Riesz potential), respectively.

For a locally integrable function $b$ on ${\mathbb{R}^{n}}$, suppose that the
commutator operator $T_{\Omega,b}$ represents a linear or a sublinear
operator, which satisfies that for any $f\in L_{1}({\mathbb{R}^{n}})$ with
compact support and $x\notin suppf$
\begin{equation}
|T_{\Omega,b}f(x)|\leq c_{0}%
%TCIMACRO{\dint \limits_{{\mathbb{R}^{n}}}}%
%BeginExpansion
{\displaystyle \int \limits_{{\mathbb{R}^{n}}}}
%EndExpansion
|b(x)-b(y)|\, \frac{|\Omega(x-y)|}{|x-y|^{n}}\,|f(y)|\,dy,\label{1}%
\end{equation}
where $c_{0}$ is independent of $f$ and $x$. Similarly, for a locally
integrable function $b$ on ${\mathbb{R}^{n}}$, suppose that the commutator
operator $T_{\Omega,b,\alpha}$, $\alpha \in \left(  0,n\right)  $ represents a
linear or a sublinear operator, which satisfies that for any $f\in
L_{1}({\mathbb{R}^{n}})$ with compact support and $x\notin suppf$
\begin{equation}
|T_{\Omega,b,\alpha}f(x)|\leq c_{0}%
%TCIMACRO{\dint \limits_{{\mathbb{R}^{n}}}}%
%BeginExpansion
{\displaystyle \int \limits_{{\mathbb{R}^{n}}}}
%EndExpansion
|b(x)-b(y)|\, \frac{|\Omega(x-y)|}{|x-y|^{n-\alpha}}\,|f(y)|\,dy,\label{1**}%
\end{equation}
where $c_{0}$ is independent of $f$ and $x$.

On the other hand, for $b\in L_{1}^{loc}\left(  {\mathbb{R}^{n}}\right)  $,
denote by $\mathrm{B}$ the multiplication operator defined by $\mathrm{B}%
f\left(  x\right)  =b\left(  x\right)  f\left(  x\right)  $ for any measurable
function $f$. If $\overline{T}_{\Omega}$ is a linear operator on some
measurable function space, then the commutator formed by $\mathrm{B}$ and
$\overline{T}_{\Omega}$ is defined by%

\[
\overline{T}_{\Omega,b}f\left(  x\right)  =[b,\overline{T}_{\Omega
}]f(x):=\left(  \mathrm{B}\overline{T}_{\Omega}-\overline{T}_{\Omega
}\mathrm{B}\right)  f\left(  x\right)  =b(x)\, \overline{T}_{\Omega
}f(x)-\overline{T}_{\Omega}(bf)(x).
\]
In 1976, Coifman et al. \cite{CRW} introduced the commutator generated by
$\overline{T}_{\Omega}$ and a locally integrable function $b$ as follows:
\begin{equation}
\lbrack b,\overline{T}_{\Omega}]f(x)\equiv b(x)\overline{T}_{\Omega
}f(x)-\overline{T}_{\Omega}(bf)(x)=p.v.\int \limits_{{\mathbb{R}^{n}}%
}[b(x)-b(y)]\frac{\Omega(x-y)}{|x-y|^{n}}f(y)dy.\label{e3}%
\end{equation}
Sometimes, the commutator defined by (\ref{e3}) is also called the commutator
in Coifman-Rocherberg-Weiss's sense, which has its root in the complex
analysis and harmonic analysis (see \cite{CRW}) and corresponding the
sublinear commutator of operator $M_{\Omega}$ is defined as follows
\[
M_{\Omega,b}\left(  f\right)  (x)=\sup_{t>0}|B(x,t)|^{-1}\int \limits_{B(x,t)}%
\left \vert b\left(  x\right)  -b\left(  y\right)  \right \vert \left \vert
\Omega \left(  x-y\right)  \right \vert |f(y)|dy.
\]
And also, the operators $[b,\overline{T}_{\Omega}]$ and $M_{\Omega,b}$ satisfy
condition (\ref{1}). Let $b$ be a locally integrable function on
${\mathbb{R}^{n}}$, then for $0<\alpha<n$ and $f$ is a suitable function, we
define the commutators generated by fractional integral and maximal operators
with rough kernel and $b$ as follows, respectively:%
\[
\lbrack b,\overline{T}_{\Omega,\alpha}]f(x)\equiv b(x)\overline{T}%
_{\Omega,\alpha}f(x)-\overline{T}_{\Omega,\alpha}(bf)(x)=\int
\limits_{{\mathbb{R}^{n}}}[b(x)-b(y)]\frac{\Omega(x-y)}{|x-y|^{n-\alpha}%
}f(y)dy,
\]%
\[
M_{\Omega,b,\alpha}\left(  f\right)  (x)=\sup_{t>0}|B(x,t)|^{-1+\frac{\alpha
}{n}}\int \limits_{B(x,t)}\left \vert b\left(  x\right)  -b\left(  y\right)
\right \vert \left \vert \Omega \left(  x-y\right)  \right \vert |f(y)|dy
\]
satisfy condition (\ref{1**}).

\begin{remark}
Suppose that $\overline{T}_{\Omega,\alpha}$, $\alpha \in \left(  0,n\right)  $
represents a linear or a sublinear operator, when $\Omega$ satisfies the
specified size conditions, the kernel of the operator $\overline{T}%
_{\Omega,\alpha}$ has no regularity, so the operator $\overline{T}%
_{\Omega,\alpha} $ is called a rough fractional integral operator. These
include the commutator operator $[b,\overline{T}_{\Omega,\alpha}]$. This also
applies to $\alpha=0$. In recent years, a variety of operators related to the
fractional integrals, C-Z operators but lacking the smoothness required in the
classical theory, have been studied (for example, see \cite{Gurbuz1, Gurbuz2}).
\end{remark}

It is worth noting that for a constant $C$, if $\overline{T}_{\Omega}$ is
linear we have,%
\begin{align*}
\lbrack b+C,\overline{T}_{\Omega}]f  & =\left(  b+C\right)  \overline
{T}_{\Omega}f-\overline{T}_{\Omega}(\left(  b+C\right)  f)\\
& =b\overline{T}_{\Omega}f+C\overline{T}_{\Omega}f-\overline{T}_{\Omega
}\left(  bf\right)  -C\overline{T}_{\Omega}f\\
& =[b,\overline{T}_{\Omega}]f.
\end{align*}
This leads one to intuitively look to spaces for which we identify functions
which differ by constants, and so it is no surprise that $b\in BMO$ (bounded
mean oscillation {space) has had the most historical significance.}

Now, let us definition of $BMO$:

\begin{definition}
$\left(  \mathbf{BMO}\text{ \textbf{function}}\right)  $ Denote the bounded
mean oscillation function space by%
\[
BMO({\mathbb{R}^{n}})=\left \{  f\in L_{1}^{loc}({\mathbb{R}^{n}}):\Vert
f\Vert_{\ast}:=\sup_{Ball:B\subset{\mathbb{R}^{n}}}\mathcal{M}_{f,B}%
<\infty \right \}  ,
\]
here and in the sequel
\[
\mathcal{M}_{f,B}:=\frac{1}{|B|}%
%TCIMACRO{\dint \limits_{B}}%
%BeginExpansion
{\displaystyle \int \limits_{B}}
%EndExpansion
|f(x)-f_{B}|dx,\qquad f_{B}=\frac{1}{|B|}%
%TCIMACRO{\dint \limits_{B}}%
%BeginExpansion
{\displaystyle \int \limits_{B}}
%EndExpansion
f(y)dy.
\]

\end{definition}

Here and henceforth, $F\approx G$ means $F\gtrsim G\gtrsim F$; while $F\gtrsim
G$ means $F\geq CG$ for a constant $C>0$; and $p^{\prime}$ and $s^{\prime}$
always denote the conjugate index of any $p>1$ and $s>1$, that is, $\frac
{1}{p^{\prime}}:=1-\frac{1}{p}$ and $\frac{1}{s^{\prime}}:=1-\frac{1}{s}$ and
also $C$ stands for a positive constant that can change its value in each
statement without explicit mention. Throughout the paper we assume that
$x\in{\mathbb{R}^{n}}$ and $r>0$ and also let $B(x,r)$ denotes $x$-centred
Euclidean ball with radius $r$, $B^{C}(x,r)$ denotes its complement and
$|B(x,r)|$ is the Lebesgue measure of the ball $B(x,r)$ and $|B(x,r)|=v_{n}%
r^{n}$, where $v_{n}=|B(0,1)|$.

\section{Background about vanishing generalized Morrey spaces}

After studying Morrey spaces in detail, researchers have passed to the concept
of generalized Morrey spaces. Firstly, motivated by the work of \cite{Morrey},
Mizuhara \cite{Miz} introduced generalized Morrey spaces $M_{p,\varphi}$ as follows:

\begin{definition}
$\left(  \text{\textbf{Generalized Morrey space}; see \cite{Miz}}\right)
$\textbf{\ }Let $\varphi(x,r)$ be a positive measurable function on
${\mathbb{R}^{n}}\times(0,\infty)$. If $0<p<\infty$, then the generalized
Morrey space $M_{p,\varphi}\equiv M_{p,\varphi}({\mathbb{R}^{n}})$ is defined
by%
\[
\left \{  f\in L_{p}^{loc}({\mathbb{R}^{n}}):\Vert f\Vert_{M_{p,\varphi}}%
=\sup \limits_{x\in{\mathbb{R}^{n}},r>0}\varphi(x,r)^{-1}\Vert f\Vert
_{L_{p}(B(x,r))}<\infty \right \}  .
\]

\end{definition}

Obviously, the above definition recover the definition of $L_{p,\lambda
}({\mathbb{R}^{n}})$ if we choose $\varphi(x,r)=r^{\frac{\lambda}{p}}$, that
is
\[
L_{p,\lambda}\left(  {\mathbb{R}^{n}}\right)  =M_{p,\varphi}\left(
{\mathbb{R}^{n}}\right)  \mid_{\varphi(x,r)=r^{\frac{\lambda}{p}}}.
\]

Everywhere in the sequel we assume that $\inf \limits_{x\in{\mathbb{R}^{n}%
},r>0}\varphi(x,r)>0$ which makes the above spaces non-trivial, since the
spaces of bounded functions are contained in these spaces. We point out that
$\varphi(x,r)$ is a measurable non-negative function and no monotonicity type
condition is imposed on these spaces.

Recently, G\"{u}rb\"{u}z \cite{Gurbuz1, Gurbuz2} has proved the boundedness of
the sublinear operators and their commutators by with rough kernels denoted by
$T_{\Omega}$, $T_{\Omega,\alpha}$, $T_{\Omega,b}$, $T_{\Omega,b,\alpha}$ on
generalized Morrey spaces $M_{p,\varphi}$, respectively.

Throughout the paper we assume that $x\in{\mathbb{R}^{n}}$ and $r>0$ and also
let $B(x,r)$ denotes the open ball centered at $x$ of radius $r$, $B^{C}(x,r)$
denotes its complement and $|B(x,r)|$ is the Lebesgue measure of the ball
$B(x,r)$ and $|B(x,r)|=v_{n}r^{n}$, where $v_{n}=|B(0,1)|$.

Now, recall that the concept of the vanishing generalized Morrey spaces
$VM_{p,\varphi}({\mathbb{R}^{n}})$ has been introduced in \cite{N. Samko}.

\begin{definition}
$\left(  \text{\textbf{Vanishing generalized Morrey space}; see \cite{N.
Samko}}\right)  $\textbf{\ }Let $\varphi(x,r)$ be a positive measurable
function on ${\mathbb{R}^{n}}\times(0,\infty)$ and $1\leq p<\infty$. The
vanishing generalized Morrey space $VM_{p,\varphi}({\mathbb{R}^{n}})$ is
defined as the spaces of functions $f\in L_{p}^{loc}({\mathbb{R}^{n}})$ such
that%
\begin{equation}
\lim \limits_{r\rightarrow0}\sup \limits_{x\in{\mathbb{R}^{n}}}\varphi
(x,r)^{-1}\Vert f\Vert_{L_{p}(B(x,r))}=0.\label{1*}%
\end{equation}

\end{definition}

Naturally, it is suitable to impose on $\varphi(x,t)$ with the following
condition:%
\begin{equation}
\lim_{t\rightarrow0}\sup \limits_{x\in{\mathbb{R}^{n}}}\frac{t^{\frac{n}{p}}%
}{\varphi(x,t)}=0,\label{2}%
\end{equation}
and%
\begin{equation}
\inf_{t>1}\sup \limits_{x\in{\mathbb{R}^{n}}}\frac{t^{\frac{n}{p}}}%
{\varphi(x,t)}>0.\label{3}%
\end{equation}
From (\ref{2}) and (\ref{3}), we easily know that the bounded functions with
compact support belong to $VM_{p,\varphi}({\mathbb{R}^{n}})$.

The space $VM_{p,\varphi}({\mathbb{R}^{n}})$ is Banach space with respect to
the norm (see, for example \cite{N. Samko})%
\begin{equation}
\Vert f\Vert_{VM_{p,\varphi}}=\sup \limits_{x\in{\mathbb{R}^{n}},r>0}%
\varphi(x,r)^{-1}\Vert f\Vert_{L_{p}(B(x,r))}.\label{4}%
\end{equation}
The spaces $VM_{p,\varphi}({\mathbb{R}^{n}})$ is closed subspaces of the
Banach spaces $M_{p,\varphi}({\mathbb{R}^{n}})$, which may be shown by
standard means.

Furthermore, we have the following embeddings:%
\[
VM_{p,\varphi}\subset M_{p,\varphi},\qquad \Vert f\Vert_{M_{p,\varphi}}%
\leq \Vert f\Vert_{VM_{p,\varphi}}.
\]

The purpose of this paper is to consider the mapping properties for
the{\ operators }$T_{\Omega}$, $T_{\Omega,\alpha}$, $T_{\Omega,b}$,
$T_{\Omega,b,\alpha}$ both on vanishing generalized Morrey spaces and
vanishing Morrey spaces, respectively. Similar results still hold for
the{\ operators }$\overline{T}_{\Omega}$, $\overline{T}_{\Omega,\alpha}$,
$[b,\overline{T}_{\Omega}]$, $[b,\overline{T}_{\Omega,\alpha}]$, $M_{\Omega
,b}$ and $M_{\Omega,b,\alpha}$, respectively{. These operators }$T_{\Omega}$,
$T_{\Omega,\alpha}$, $T_{\Omega,b}$, $T_{\Omega,b,\alpha}$ have not also been
studied so far both on vanishing generalized Morrey spaces and vanishing
Morrey spaces and this paper seems to be the first in this direction.

\section{Main Results}

\begin{theorem}
\label{teo1}Let $\Omega \in L_{s}(S^{n-1})$, $1<s\leq \infty$, be homogeneous of
degree zero, and $1<p<\infty$. Let $T_{\Omega}$ be a sublinear operator
satisfying condition (\ref{e1}). Let for $s^{\prime}\leq p$, the pair
$(\varphi_{1},\varphi_{2})$ satisfies conditions (\ref{2})-(\ref{3}) and
\begin{equation}
c_{\delta}:=%
%TCIMACRO{\dint \limits_{\delta}^{\infty}}%
%BeginExpansion
{\displaystyle \int \limits_{\delta}^{\infty}}
%EndExpansion
\sup_{x\in{\mathbb{R}^{n}}}\varphi_{1}\left(  x,t\right)  t^{-\frac{n}{p}%
-1}dt<\infty \label{6}%
\end{equation}
for every $\delta>0$, and
\begin{equation}
\int \limits_{r}^{\infty}\frac{\varphi_{1}(x,t)}{t^{\frac{n}{p}+1}}dt\leq
C_{0}\frac{\varphi_{2}(x,r)}{r^{\frac{n}{p}}},\label{7}%
\end{equation}
and for $1<p<s$ the pair $(\varphi_{1},\varphi_{2})$ satisfies conditions
(\ref{2})-(\ref{3}) and also%
\begin{equation}
c_{\delta^{\prime}}:=%
%TCIMACRO{\dint \limits_{\delta^{\prime}}^{\infty}}%
%BeginExpansion
{\displaystyle \int \limits_{\delta^{\prime}}^{\infty}}
%EndExpansion
\sup_{x\in{\mathbb{R}^{n}}}\varphi_{1}\left(  x,t\right)  t^{-\frac{n}%
{p}+\frac{n}{s}-1}dt<\infty \label{8}%
\end{equation}
for every $\delta^{\prime}>0$, and%
\begin{equation}
\int \limits_{r}^{\infty}\frac{\varphi_{1}(x,t)}{t^{\frac{n}{p}-\frac{n}{s}+1}%
}dt\leq C_{0}\, \frac{\varphi_{2}(x,r)}{r^{\frac{n}{p}-\frac{n}{s}}},\label{9}%
\end{equation}
where $C_{0}$ does not depend on $x\in{\mathbb{R}^{n}}$ and $r>0$.

Then the operator $T_{\Omega}$ is bounded from $VM_{p,\varphi_{1}}$ to
$VM_{p,\varphi_{2}}$ for $p>1$. Moreover, we have for $p>1$%
\[
\left \Vert T_{\Omega}f\right \Vert _{VM_{p,\varphi_{2}}}\lesssim \left \Vert
f\right \Vert _{VM_{p,\varphi_{1}}}.
\]

\end{theorem}

\begin{proof}
Let $1<p<\infty$ and $s^{\prime}\leq p$. The estimation of the norm of the
operator, that is, the boundedness in the vanishing generalized Morrey space
follows from Lemma 2.1. in \cite{Gurbuz1} and condition (\ref{7})
\begin{align*}
\Vert T_{\Omega}f\Vert_{VM_{p,\varphi_{2}}}  & =\sup \limits_{x\in
{\mathbb{R}^{n}},r>0}\varphi_{2}(x,r)^{-1}\Vert T_{\Omega}f\Vert
_{L_{p}(B(x,r))}\\
& \lesssim \sup_{x\in{\mathbb{R}^{n},}r>0}\varphi_{2}\left(  x,r\right)
^{-1}r^{\frac{n}{p}}\int \limits_{r}^{\infty}\left \Vert f\right \Vert
_{L_{p}\left(  B\left(  x,t\right)  \right)  }\frac{dt}{t^{\frac{n}{p}+1}}\\
& \lesssim \sup_{x\in{\mathbb{R}^{n},}r>0}\varphi_{2}\left(  x,r\right)
^{-1}r^{\frac{n}{p}}\int \limits_{r}^{\infty}\varphi_{1}\left(  x,t\right)
\left[  \varphi_{1}\left(  x,t\right)  ^{-1}\left \Vert f\right \Vert
_{L_{p}\left(  B\left(  x,t\right)  \right)  }\right]  \frac{dt}{t^{\frac
{n}{p}+1}}\\
& \lesssim \left \Vert f\right \Vert _{VM_{p,\varphi_{1}}}\sup_{x\in
{\mathbb{R}^{n},}r>0}\varphi_{2}\left(  x,r\right)  ^{-1}r^{\frac{n}{p}}%
\int \limits_{r}^{\infty}\varphi_{1}\left(  x,t\right)  \frac{dt}{t^{\frac
{n}{p}+1}}\\
& \lesssim \left \Vert f\right \Vert _{VM_{p,\varphi_{1}}}.
\end{align*}

So it is sufficient to prove that%
\begin{equation}
\lim \limits_{r\rightarrow0}\sup \limits_{x\in{\mathbb{R}^{n}}}\varphi
_{1}(x,r)^{-1}\, \left \Vert f\right \Vert _{L_{p}\left(  B\left(  x,r\right)
\right)  }=0\text{ implies }\lim \limits_{r\rightarrow0}\sup \limits_{x\in
{\mathbb{R}^{n}}}\varphi_{2}(x,r)^{-1}\, \left \Vert T_{\Omega}f\right \Vert
_{L_{p}\left(  B\left(  x,r\right)  \right)  }=0.\label{12}%
\end{equation}

To show that $\sup \limits_{x\in{\mathbb{R}^{n}}}\varphi_{2}(x,r)^{-1}\,
\left \Vert T_{\Omega}f\right \Vert _{L_{p}\left(  B\left(  x,r\right)  \right)
}<\epsilon$ for small $r$, we split the right hand side of (2.1) in Lemma 2.1.
in \cite{Gurbuz1}:%
\begin{equation}
\varphi_{2}(x,r)^{-1}\, \left \Vert T_{\Omega}f\right \Vert _{L_{p}\left(
B\left(  x,r\right)  \right)  }\leq C\left[  I_{\delta_{0}}\left(  x,r\right)
+J_{\delta_{0}}\left(  x,r\right)  \right]  ,\label{13}%
\end{equation}
where $\delta_{0}>0$ (we may take $\delta_{0}<1$), and
\[
I_{\delta_{0}}\left(  x,r\right)  :=\frac{r^{\frac{n}{p}}}{\varphi_{2}%
(x,r)}\left(
%TCIMACRO{\dint \limits_{r}^{\delta_{0}}}%
%BeginExpansion
{\displaystyle \int \limits_{r}^{\delta_{0}}}
%EndExpansion
\varphi_{1}\left(  x,t\right)  t^{-\frac{n}{p}-1}\left(  \varphi_{1}\left(
x,t\right)  ^{-1}\left \Vert f\right \Vert _{L_{p}\left(  B\left(  x,t\right)
\right)  }\right)  dt\right)  ,
\]
and%
\[
J_{\delta_{0}}\left(  x,r\right)  :=\frac{r^{\frac{n}{p}}}{\varphi_{2}%
(x,r)}\left(
%TCIMACRO{\dint \limits_{\delta_{0}}^{\infty}}%
%BeginExpansion
{\displaystyle \int \limits_{\delta_{0}}^{\infty}}
%EndExpansion
\varphi_{1}\left(  x,t\right)  t^{-\frac{n}{p}-1}\left(  \varphi_{1}\left(
x,t\right)  ^{-1}\left \Vert f\right \Vert _{L_{p}\left(  B\left(  x,t\right)
\right)  }\right)  dt\right)
\]
and $r<\delta_{0}$. Now we choose any fixed $\delta_{0}>0$ such that%
\[
\sup \limits_{x\in{\mathbb{R}^{n}}}\varphi_{1}\left(  x,t\right)
^{-1}\left \Vert f\right \Vert _{L_{p}\left(  B\left(  x,t\right)  \right)
}<\frac{\epsilon}{2CC_{0}},
\]
where $C$ and $C_{0}$ are constants from (\ref{7}) and (\ref{13}). This allows
to estimate the first term uniformly in $r\in \left(  0,\delta_{0}\right)  :$%
\[
\sup \limits_{x\in{\mathbb{R}^{n}}}CI_{\delta_{0}}\left(  x,r\right)
<\frac{\epsilon}{2},\qquad0<r<\delta_{0}.
\]

The estimation of the second term may be obtained by choosing $r$ sufficiently
small. Indeed, by (\ref{2}) we have%
\[
J_{\delta_{0}}\left(  x,r\right)  \leq c_{\delta_{0}}\left \Vert f\right \Vert
_{VM_{p,\varphi}}\frac{r^{\frac{n}{p}}}{\varphi \left(  x,r\right)  },
\]
where $c_{\delta_{0}}$ is the constant from (\ref{6}). Then, by (\ref{2}) it
suffices to choose $r$ small enough such that
\[
\sup \limits_{x\in{\mathbb{R}^{n}}}\frac{r^{\frac{n}{p}}}{\varphi(x,r)}%
\leq \frac{\epsilon}{2c_{\delta_{0}}\left \Vert f\right \Vert _{VM_{p,\varphi}}},
\]
which completes the proof of (\ref{12}).

For the case of $1<p<s$, we can also use the same method, so we omit the
details, which completes the proof.
\end{proof}

\begin{remark}
Conditions (\ref{6}) and (\ref{8}) are not needed in the case when
$\varphi(x,r)$ does not depend on $x$, since (\ref{6}) follows from (\ref{7})
and similarly, (\ref{8}) follows from (\ref{9}) in this case.
\end{remark}

\begin{corollary}
Under the conditions of Theorem \ref{teo1}, the operators $M_{\Omega}$ and
$\overline{T}_{\Omega}$ are bounded from $VM_{p,\varphi_{1}}$ to
$VM_{p,\varphi_{2}}$.
\end{corollary}

\begin{corollary}
\label{Corollary0}Let $\Omega \in L_{s}(S^{n-1})$, $1<s\leq \infty$, be
homogeneous of degree zero satisfying condition (\ref{0}). Let $0<\lambda<n$,
$1<p<\infty$. Let $T_{\Omega}$ be a sublinear operator satisfying condition
(\ref{e1}). Then for $s^{\prime}\leq p$ or $p<s$, we have%
\[
\left \Vert T_{\Omega}f\right \Vert _{VM_{p,\lambda}}\lesssim \left \Vert
f\right \Vert _{VM_{p,\lambda}}.
\]

\end{corollary}

\begin{proof}
Let $1<p<\infty$ and $s^{\prime}\leq p$. By using $\varphi_{1}\left(
x,r\right)  =\varphi_{2}\left(  x,r\right)  =r^{\frac{\lambda}{p}}$ in the
proof of Theorem \ref{teo1} and condition (\ref{7}), we get%
\begin{align*}
\Vert T_{\Omega}f\Vert_{VM_{p,\lambda}}  & \lesssim \sup_{x\in{\mathbb{R}^{n}%
,}r>0}r^{-\frac{\lambda}{p}}r^{\frac{n}{p}}\int \limits_{r}^{\infty}\varphi
_{1}\left(  x,t\right)  t^{-\frac{n}{p}-1}\left(  \varphi_{1}\left(
x,t\right)  ^{-1}\left \Vert f\right \Vert _{L_{p}\left(  B\left(  x,t\right)
\right)  }\right)  dt\\
& \lesssim \left \Vert f\right \Vert _{VM_{p,\lambda}}\sup_{x\in{\mathbb{R}^{n}%
,}r>0}r^{\frac{n-\lambda}{p}}\int \limits_{r}^{\infty}r^{\frac{\lambda}{p}%
}\frac{dt}{t^{\frac{n}{p}+1}}\\
& \lesssim \left \Vert f\right \Vert _{VM_{p,\lambda}},
\end{align*}
for the case of $p<s$, we can also use the same method, so we omit the details.
\end{proof}

\begin{corollary}
Under the conditions of Corollary \ref{Corollary0}, the operators $M_{\Omega}$
and $\overline{T}_{\Omega}$ are bounded on $VM_{p,\lambda}\left(
{\mathbb{R}^{n}}\right)  $.
\end{corollary}

\begin{theorem}
\label{teo2} Let $\Omega \in L_{s}(S^{n-1})$, $1<s\leq \infty$, be homogeneous
of degree zero. Let $0<\alpha<n$, $1<p<\frac{n}{\alpha}$ and $\frac{1}%
{q}=\frac{1}{p}-\frac{\alpha}{n}$. Let $T_{\Omega,\alpha}$ be a sublinear
operator satisfying condition (\ref{e2}). Let for $s^{\prime}\leq p$ the pair
$(\varphi_{1},\varphi_{2})$ satisfies conditions (\ref{2})-(\ref{3}) and
\begin{equation}
c_{\delta}:=%
%TCIMACRO{\dint \limits_{\delta}^{\infty}}%
%BeginExpansion
{\displaystyle \int \limits_{\delta}^{\infty}}
%EndExpansion
\sup_{x\in{\mathbb{R}^{n}}}\varphi_{1}\left(  x,t\right)  \frac{dt}%
{t^{\frac{n}{q}+1}}<\infty \label{6**}%
\end{equation}
for every $\delta>0$, and
\begin{equation}
\int \limits_{r}^{\infty}\varphi_{1}\left(  x,t\right)  \frac{dt}{t^{\frac
{n}{q}+1}}\leq C_{0}\frac{\varphi_{2}(x,r)}{r^{\frac{n}{q}}},\label{7**}%
\end{equation}
and for $q<s$ the pair $(\varphi_{1},\varphi_{2})$ satisfies conditions
(\ref{2})-(\ref{3}) and also%
\begin{equation}
c_{\delta^{\prime}}:=%
%TCIMACRO{\dint \limits_{\delta^{\prime}}^{\infty}}%
%BeginExpansion
{\displaystyle \int \limits_{\delta^{\prime}}^{\infty}}
%EndExpansion
\sup_{x\in{\mathbb{R}^{n}}}\varphi_{1}(x,t)\frac{dt}{t^{\frac{n}{q}-\frac
{n}{s}+1}}<\infty \label{8**}%
\end{equation}
for every $\delta^{\prime}>0$, and%
\begin{equation}
\int \limits_{r}^{\infty}\varphi_{1}(x,t)\frac{dt}{t^{\frac{n}{q}-\frac{n}%
{s}+1}}\leq C_{0}\frac{\varphi_{2}(x,r)}{r^{\frac{n}{q}-\frac{n}{s}}%
},\label{9**}%
\end{equation}
where $C_{0}$ does not depend on $x\in{\mathbb{R}^{n}}$ and $r>0$.

Then the operator $T_{\Omega,\alpha}$ is bounded from $VM_{p,\varphi_{1}}$ to
$VM_{q,\varphi_{2}}$ for $p>1$. Moreover, we have for $p>1$%
\[
\left \Vert T_{\Omega,\alpha}f\right \Vert _{VM_{q,\varphi_{2}}}\lesssim
\left \Vert f\right \Vert _{VM_{p,\varphi_{1}}}.
\]

\end{theorem}

\begin{proof}
Similar to the proof of Theorem \ref{teo1}, let $s^{\prime}\leq p$. The
estimation of the norm of the operator follows from Lemma 3 in \cite{Gurbuz2}
and condition (\ref{7**})
\begin{align*}
\Vert T_{\Omega,\alpha}f\Vert_{VM_{q,\varphi_{2}}}  & =\sup \limits_{x\in
{\mathbb{R}^{n}},r>0}\varphi_{2}(x,r)^{-1}\Vert T_{\Omega,\alpha}f\Vert
_{L_{q}(B(x,r))}\\
& \lesssim \sup_{x\in{\mathbb{R}^{n},}r>0}\varphi_{2}\left(  x,r\right)
^{-1}r^{\frac{n}{q}}\int \limits_{r}^{\infty}\left \Vert f\right \Vert
_{L_{p}\left(  B\left(  x,t\right)  \right)  }\frac{dt}{t^{\frac{n}{q}+1}}\\
& \lesssim \sup_{x\in{\mathbb{R}^{n},}r>0}\varphi_{2}\left(  x,r\right)
^{-1}r^{\frac{n}{q}}\int \limits_{r}^{\infty}\varphi_{1}\left(  x,t\right)
\left[  \varphi_{1}\left(  x,t\right)  ^{-1}\left \Vert f\right \Vert
_{L_{p}\left(  B\left(  x,t\right)  \right)  }\right]  \frac{dt}{t^{\frac
{n}{q}+1}}\\
& \lesssim \left \Vert f\right \Vert _{VM_{p,\varphi_{1}}}\sup_{x\in
{\mathbb{R}^{n},}r>0}\varphi_{2}\left(  x,r\right)  ^{-1}r^{\frac{n}{q}}%
\int \limits_{r}^{\infty}\varphi_{1}\left(  x,t\right)  \frac{dt}{t^{\frac
{n}{q}+1}}\\
& \lesssim \left \Vert f\right \Vert _{VM_{p,\varphi_{1}}}.
\end{align*}
Thus we only have to prove that%
\begin{equation}
\lim \limits_{r\rightarrow0}\sup \limits_{x\in{\mathbb{R}^{n}}}\varphi
_{1}(x,r)^{-1}\, \left \Vert f\right \Vert _{L_{p}\left(  B\left(  x,r\right)
\right)  }=0\text{ implies }\lim \limits_{r\rightarrow0}\sup \limits_{x\in
{\mathbb{R}^{n}}}\varphi_{2}(x,r)^{-1}\, \left \Vert T_{\Omega,\alpha
}f\right \Vert _{L_{q}\left(  B\left(  x,r\right)  \right)  }=0.\label{12**}%
\end{equation}

To show that $\sup \limits_{x\in{\mathbb{R}^{n}}}\frac{\left \Vert
T_{\Omega,\alpha}f\right \Vert _{L_{q}\left(  B\left(  x,r\right)  \right)  }%
}{\varphi_{2}(x,r)}<\epsilon$ for small $r$, we split the right-hand side of
(2.1) in Lemma 3 in \cite{Gurbuz2}:%
\[
\frac{r^{-\frac{n}{q}}\left \Vert T_{\Omega,\alpha}f\right \Vert _{L_{q}\left(
B\left(  x,r\right)  \right)  }}{\varphi_{2}(x,r)}\leq C\left[  I_{\delta_{0}%
}\left(  x,r\right)  +J_{\delta_{0}}\left(  x,r\right)  \right]  ,
\]
where $\delta_{0}>0$ (we may take $\delta_{0}<1$), and
\[
I_{\delta_{0}}\left(  x,r\right)  :=\frac{r^{\frac{n}{q}}}{\varphi_{2}(x,r)}%
%TCIMACRO{\dint \limits_{r}^{\delta_{0}}}%
%BeginExpansion
{\displaystyle \int \limits_{r}^{\delta_{0}}}
%EndExpansion
\varphi_{1}\left(  x,t\right)  t^{-\frac{n}{q}-1}\left(  \varphi_{1}\left(
x,t\right)  ^{-1}\left \Vert f\right \Vert _{L_{p}\left(  B\left(  x,t\right)
\right)  }\right)  dt,
\]
and%
\[
J_{\delta_{0}}\left(  x,r\right)  :=\frac{r^{\frac{n}{q}}}{\varphi_{2}(x,r)}%
%TCIMACRO{\dint \limits_{\delta_{0}}^{\infty}}%
%BeginExpansion
{\displaystyle \int \limits_{\delta_{0}}^{\infty}}
%EndExpansion
\varphi_{1}\left(  x,t\right)  t^{-\frac{n}{q}-1}\left(  \varphi_{1}\left(
x,t\right)  ^{-1}\left \Vert f\right \Vert _{L_{p}\left(  B\left(  x,t\right)
\right)  }\right)  dt,
\]
and $r<\delta_{0}$ and the rest of the proof is the same as the proof of
Theorem \ref{teo1}. Thus, we can prove that (\ref{12**}).

For the case of $q<s$, we can also use the same method, so we omit the
details, which completes the proof.
\end{proof}

\begin{remark}
Conditions (\ref{6**}) and (\ref{8**}) are not needed in the case when
$\varphi(x,r)$ does not depend on $x$, since (\ref{6**}) follows from
(\ref{7**}) and similarly, (\ref{8**}) follows from (\ref{9**}) in this case.
\end{remark}

\begin{corollary}
Under the conditions of Theorem \ref{teo2}, the operators $M_{\Omega,\alpha}$
and $\overline{T}_{\Omega,\alpha}$ are bounded from $VM_{p,\varphi_{1}} $ to
$VM_{q,\varphi_{2}}$.
\end{corollary}

\begin{corollary}
\label{Corollary1}Let $\Omega \in L_{s}(S^{n-1})$, $1<s\leq \infty$, be
homogeneous of degree zero. Let $0<\alpha,\lambda<n$, $1<p<\frac{n-\lambda
}{\alpha}$, $\frac{1}{p}-\frac{1}{q}=\frac{\alpha}{n}$ and $\frac{\lambda}%
{p}=\frac{\mu}{q}$. Let $T_{\Omega,\alpha}$ be a sublinear operator satisfying
condition (\ref{e2}). Then for $s^{\prime}\leq p$ or $q<s$, we have%
\[
\left \Vert T_{\Omega,\alpha}f\right \Vert _{VM_{q,\mu}}\lesssim \left \Vert
f\right \Vert _{VM_{p,\lambda}}.
\]

\end{corollary}

\begin{proof}
Let $s^{\prime}\leq p$. By using $\varphi_{1}\left(  x,r\right)
=r^{\frac{\lambda}{p}}$ and $\varphi_{2}\left(  x,r\right)  =r^{\frac{\mu}{q}%
}$ in the proof of Theorem \ref{teo2} and condition (\ref{7**}), it follows
that%
\begin{align*}
\Vert T_{\Omega,\alpha}f\Vert_{VM_{q,\mu}}  & \lesssim \sup_{x\in
{\mathbb{R}^{n},}r>0}r^{-\frac{\mu}{q}}r^{\frac{n}{q}}\int \limits_{r}^{\infty
}\varphi_{1}\left(  x,t\right)  t^{-\frac{n}{q}-1}\left(  \varphi_{1}\left(
x,t\right)  ^{-1}\left \Vert f\right \Vert _{L_{p}\left(  B\left(  x,t\right)
\right)  }\right)  dt\\
& \lesssim \left \Vert f\right \Vert _{VM_{p,\lambda}}\sup_{x\in{\mathbb{R}^{n}%
,}r>0}r^{\frac{n-\mu}{q}}\int \limits_{r}^{\infty}r^{\frac{\lambda}{p}}%
\frac{dt}{t^{\frac{n}{q}+1}}\\
& \lesssim \left \Vert f\right \Vert _{VM_{p,\lambda}},
\end{align*}
for the case of $q<s$, we can also use the same method, so we omit the
details, which completes the proof.
\end{proof}

\begin{corollary}
Under the conditions of Corollary \ref{Corollary1}, the operators
$M_{\Omega,\alpha}$ and $\overline{T}_{\Omega,\alpha}$ are bounded from
$VM_{p,\lambda}$ to $VM_{q,\mu}$.
\end{corollary}

Now below, we obtain the boundedness of operators both $T_{\Omega,b}$ and
$T_{\Omega,b,\alpha}$ on the vanishing generalized Morrey spaces
$VM_{p,\varphi}$.

\begin{theorem}
\label{teo3}Let $\Omega \in L_{s}(S^{n-1})$,$1<s\leq \infty$, be homogeneous of
degree zero. Let $1<p<\infty$ and $b\in BMO\left(
%TCIMACRO{\U{211d} }%
%BeginExpansion
\mathbb{R}
%EndExpansion
^{n}\right)  $. Let $T_{\Omega,b}$ is a sublinear operator satisfying
condition (\ref{1}). Let for $s^{\prime}\leq p$ the pair $(\varphi_{1}%
,\varphi_{2})$ satisfies conditions (\ref{2})-(\ref{3}) and
\begin{equation}
c_{\delta}:=%
%TCIMACRO{\dint \limits_{\delta}^{\infty}}%
%BeginExpansion
{\displaystyle \int \limits_{\delta}^{\infty}}
%EndExpansion
\left(  1+\ln \frac{t}{r}\right)  \sup_{x\in{\mathbb{R}^{n}}}\varphi_{1}\left(
x,t\right)  t^{-\frac{n}{p}-1}dt<\infty \label{6*}%
\end{equation}
for every $\delta>0$, and
\begin{equation}
\int \limits_{r}^{\infty}\left(  1+\ln \frac{t}{r}\right)  \frac{\varphi
_{1}(x,t)}{t^{\frac{n}{p}+1}}dt\leq C_{0}\frac{\varphi_{2}(x,r)}{r^{\frac
{n}{p}}},\label{7*}%
\end{equation}
and for $p<s$ the pair $(\varphi_{1},\varphi_{2})$ satisfies conditions
(\ref{2})-(\ref{3}) and also%
\begin{equation}
c_{\delta^{\prime}}:=%
%TCIMACRO{\dint \limits_{\delta^{\prime}}^{\infty}}%
%BeginExpansion
{\displaystyle \int \limits_{\delta^{\prime}}^{\infty}}
%EndExpansion
\left(  1+\ln \frac{t}{r}\right)  \sup_{x\in{\mathbb{R}^{n}}}\varphi_{1}\left(
x,t\right)  t^{-\frac{n}{p}+\frac{n}{s}-1}dt<\infty \label{8*}%
\end{equation}
for every $\delta^{\prime}>0$, and%
\begin{equation}
\int \limits_{r}^{\infty}\left(  1+\ln \frac{t}{r}\right)  \frac{\varphi
_{1}(x,t)}{t^{\frac{n}{p}-\frac{n}{s}+1}}dt\leq C_{0}\, \frac{\varphi
_{2}(x,r)}{r^{\frac{n}{p}-\frac{n}{s}}},\label{9*}%
\end{equation}
where $C_{0}$ does not depend on $x\in{\mathbb{R}^{n}}$ and $r>0$.

Then the operator $T_{\Omega,b}$ is bounded from $VM_{p,\varphi_{1}}$ to
$VM_{p,\varphi_{2}}$. Moreover,%
\[
\left \Vert T_{\Omega,b}f\right \Vert _{VM_{p,\varphi_{2}}}\lesssim \left \Vert
b\right \Vert _{\ast}\left \Vert f\right \Vert _{VM_{p,\varphi_{1}}}.
\]

\end{theorem}

\begin{proof}
The proof follows more or less the same lines as for Theorem \ref{teo1}, but
now the arguments are different due to the necessity to introduce the
logarithmic factor into the assumptions. Let $s^{\prime}\leq p$. The
estimation of the norm of the operator, that is, the boundedness in the
vanishing generalized Morrey space follows from Lemma 2.2. in \cite{Gurbuz1}
and condition (\ref{7*})
\begin{align*}
\Vert T_{\Omega,b}f\Vert_{VM_{p,\varphi_{2}}}  & =\sup \limits_{x\in
{\mathbb{R}^{n}},r>0}\varphi_{2}(x,r)^{-1}\Vert T_{\Omega,b}f\Vert
_{L_{p}(B(x,r))}\\
& \lesssim \left \Vert b\right \Vert _{\ast}\sup_{x\in{\mathbb{R}^{n},}%
r>0}\varphi_{2}\left(  x,r\right)  ^{-1}r^{\frac{n}{p}}\int \limits_{r}%
^{\infty}\left(  1+\ln \frac{t}{r}\right)  \left \Vert f\right \Vert
_{L_{p}\left(  B\left(  x,t\right)  \right)  }\frac{dt}{t^{\frac{n}{p}+1}}\\
& \lesssim \left \Vert b\right \Vert _{\ast}\sup_{x\in{\mathbb{R}^{n},}%
r>0}\varphi_{2}\left(  x,r\right)  ^{-1}r^{\frac{n}{p}}\int \limits_{r}%
^{\infty}\left(  1+\ln \frac{t}{r}\right)  \varphi_{1}\left(  x,t\right)
\left[  \varphi_{1}\left(  x,t\right)  ^{-1}\left \Vert f\right \Vert
_{L_{p}\left(  B\left(  x,t\right)  \right)  }\right]  \frac{dt}{t^{\frac
{n}{p}+1}}\\
& \lesssim \left \Vert b\right \Vert _{\ast}\left \Vert f\right \Vert
_{VM_{p,\varphi_{1}}}\sup_{x\in{\mathbb{R}^{n},}r>0}\varphi_{2}\left(
x,r\right)  ^{-1}r^{\frac{n}{p}}\int \limits_{r}^{\infty}\left(  1+\ln \frac
{t}{r}\right)  \varphi_{1}\left(  x,t\right)  \frac{dt}{t^{\frac{n}{p}+1}}\\
& \lesssim \left \Vert b\right \Vert _{\ast}\left \Vert f\right \Vert
_{VM_{p,\varphi_{1}}}.
\end{align*}

So we only have to prove that%
\begin{equation}
\lim \limits_{r\rightarrow0}\sup \limits_{x\in{\mathbb{R}^{n}}}\varphi
_{1}(x,r)^{-1}\, \left \Vert f\right \Vert _{L_{p}\left(  B\left(  x,r\right)
\right)  }=0\text{ implies }\lim \limits_{r\rightarrow0}\sup \limits_{x\in
{\mathbb{R}^{n}}}\varphi_{2}(x,r)^{-1}\, \left \Vert T_{\Omega,b}f\right \Vert
_{L_{p}\left(  B\left(  x,r\right)  \right)  }=0.\label{12*}%
\end{equation}

To show that $\sup \limits_{x\in{\mathbb{R}^{n}}}\varphi_{2}(x,r)^{-1}\,
\left \Vert T_{\Omega,b}f\right \Vert _{L_{p}\left(  B\left(  x,r\right)
\right)  }<\epsilon$ for small $r$, we split the right-hand side of the first
inequality in Lemma 2.2. in \cite{Gurbuz1}:%
\begin{equation}
\varphi_{2}(x,r)^{-1}\, \left \Vert T_{\Omega}f\right \Vert _{L_{p}\left(
B\left(  x,r\right)  \right)  }\leq C\left[  I_{\delta_{0}}\left(  x,r\right)
+J_{\delta_{0}}\left(  x,r\right)  \right]  ,\label{13*}%
\end{equation}
where $\delta_{0}>0$ (we may take $\delta_{0}<1$), and
\[
I_{\delta_{0}}\left(  x,r\right)  :=\left \Vert b\right \Vert _{\ast}%
\frac{r^{\frac{n}{p}}}{\varphi_{2}(x,r)}\left(
%TCIMACRO{\dint \limits_{r}^{\delta_{0}}}%
%BeginExpansion
{\displaystyle \int \limits_{r}^{\delta_{0}}}
%EndExpansion
\left(  1+\ln \frac{t}{r}\right)  \varphi_{1}\left(  x,t\right)  t^{-\frac
{n}{p}-1}\left(  \varphi_{1}\left(  x,t\right)  ^{-1}\left \Vert f\right \Vert
_{L_{p}\left(  B\left(  x,t\right)  \right)  }\right)  dt\right)  ,
\]
and%
\[
J_{\delta_{0}}\left(  x,r\right)  :=\left \Vert b\right \Vert _{\ast}%
\frac{r^{\frac{n}{p}}}{\varphi_{2}(x,r)}\left(
%TCIMACRO{\dint \limits_{\delta_{0}}^{\infty}}%
%BeginExpansion
{\displaystyle \int \limits_{\delta_{0}}^{\infty}}
%EndExpansion
\left(  1+\ln \frac{t}{r}\right)  \varphi_{1}\left(  x,t\right)  t^{-\frac
{n}{p}-1}\left(  \varphi_{1}\left(  x,t\right)  ^{-1}\left \Vert f\right \Vert
_{L_{p}\left(  B\left(  x,t\right)  \right)  }\right)  dt\right)
\]
and $r<\delta_{0}$. Now we choose any fixed $\delta_{0}>0$ such that%
\[
\sup \limits_{x\in{\mathbb{R}^{n}}}\varphi_{1}\left(  x,t\right)
^{-1}\left \Vert f\right \Vert _{L_{p}\left(  B\left(  x,t\right)  \right)
}<\frac{\epsilon}{2CC_{0}},
\]
where $C$ and $C_{0}$ are constants from (\ref{7*}) and (\ref{13*}). This
allows to estimate the first term uniformly in $r\in \left(  0,\delta
_{0}\right)  $:%
\[
\left \Vert b\right \Vert _{\ast}\sup \limits_{x\in{\mathbb{R}^{n}}}%
CI_{\delta_{0}}\left(  x,r\right)  <\frac{\epsilon}{2},\qquad0<r<\delta_{0}.
\]

The estimation of the second term may be obtained by choosing $r$ sufficiently
small. Indeed, by (\ref{2}) we have%
\[
J_{\delta_{0}}\left(  x,r\right)  \leq \left \Vert b\right \Vert _{\ast}%
c_{\delta_{0}}\left \Vert f\right \Vert _{VM_{p,\varphi}}\frac{r^{\frac{n}{p}}%
}{\varphi \left(  x,r\right)  },
\]
where $c_{\delta_{0}}$ is the constant from (\ref{6*}). Then, by (\ref{2}) it
suffices to choose $r$ small enough such that
\[
\sup \limits_{x\in{\mathbb{R}^{n}}}\frac{r^{\frac{n}{p}}}{\varphi(x,r)}%
\leq \frac{\epsilon}{2\left \Vert b\right \Vert _{\ast}c_{\delta_{0}}\left \Vert
f\right \Vert _{VM_{p,\varphi}}},
\]
which completes the proof of (\ref{12*}).

For the case of $p<s$, we can also use the same method, so we omit the details.
\end{proof}

\begin{remark}
Conditions (\ref{6*}) and (\ref{8*}) are not needed in the case when
$\varphi(x,r)$ does not depend on $x$, since (\ref{6*}) follows from
(\ref{7*}) and similarly, (\ref{8*}) follows from (\ref{9*}) in this case.
\end{remark}

\begin{corollary}
Under the conditions of Theorem \ref{teo3}, the operators $M_{\Omega,b}$ and
$[b,\overline{T}_{\Omega}]$ are bounded from $VM_{p,\varphi_{1}}$ to
$VM_{p,\varphi_{2}}$.
\end{corollary}

\begin{corollary}
\label{Corollary0*}Let $\Omega \in L_{s}(S^{n-1})$, $1<s\leq \infty$, be
homogeneous of degree zero satisfying condition (\ref{0}). Let $0<\lambda<n$,
$1<p<\infty$. Let $1<p<\infty$ and $b\in BMO\left(
%TCIMACRO{\U{211d} }%
%BeginExpansion
\mathbb{R}
%EndExpansion
^{n}\right)  $. Let $T_{\Omega,b}$ be a sublinear operator satisfying
condition (\ref{1}). Then for $s^{\prime}\leq p$ or $p<s$, we have%
\[
\left \Vert T_{\Omega,b}f\right \Vert _{VM_{p,\lambda}}\lesssim \left \Vert
b\right \Vert _{\ast}\left \Vert f\right \Vert _{VM_{p,\lambda}}.
\]

\end{corollary}

\begin{proof}
Let $1<p<\infty$, $b\in BMO\left(
%TCIMACRO{\U{211d} }%
%BeginExpansion
\mathbb{R}
%EndExpansion
^{n}\right)  $ and $s^{\prime}\leq p$. By using $\varphi_{1}\left(
x,r\right)  =\varphi_{2}\left(  x,r\right)  =r^{\frac{\lambda}{p}}$ in the
proof of Theorem \ref{teo3} and condition (\ref{7*}), we get%
\begin{align*}
\Vert T_{\Omega,b}f\Vert_{VM_{p,\lambda}}  & \lesssim \left \Vert b\right \Vert
_{\ast}\sup_{x\in{\mathbb{R}^{n},}r>0}r^{-\frac{\lambda}{p}}r^{\frac{n}{p}%
}\int \limits_{r}^{\infty}\left(  1+\ln \frac{t}{r}\right)  \varphi_{1}\left(
x,t\right)  t^{-\frac{n}{p}-1}\left(  \varphi_{1}\left(  x,t\right)
^{-1}\left \Vert f\right \Vert _{L_{p}\left(  B\left(  x,t\right)  \right)
}\right)  dt\\
& \lesssim \left \Vert b\right \Vert _{\ast}\left \Vert f\right \Vert
_{VM_{p,\lambda}}\sup_{x\in{\mathbb{R}^{n},}r>0}r^{\frac{n-\lambda}{p}}%
\int \limits_{r}^{\infty}\left(  1+\ln \frac{t}{r}\right)  r^{\frac{\lambda}{p}%
}\frac{dt}{t^{\frac{n}{p}+1}}\\
& \lesssim \left \Vert b\right \Vert _{\ast}\left \Vert f\right \Vert
_{VM_{p,\lambda}},
\end{align*}
for the case of $p<s$, we can also use the same method, so we omit the details.
\end{proof}

\begin{corollary}
Under the conditions of Corollary \ref{Corollary0*}, the operators
$M_{\Omega,b}$ and $[b,\overline{T}_{\Omega}]$ are bounded on $VM_{p,\lambda
}\left(  {\mathbb{R}^{n}}\right)  $.
\end{corollary}

\begin{theorem}
\label{teo4}Let $\Omega \in L_{s}(S^{n-1})$, $1<s\leq \infty$, be homogeneous of
degree zero. Let $1<p<\infty$, $0<\alpha<\frac{n}{p}$, $\frac{1}{q}=\frac
{1}{p}-\frac{\alpha}{n}$ and $b\in BMO\left(  {\mathbb{R}^{n}}\right)  $. Let
$T_{\Omega,b,\alpha}$ be a sublinear operator satisfying condition
(\ref{1**}). Let for $s^{\prime}\leq p$ the pair $(\varphi_{1},\varphi_{2})$
satisfies conditions (\ref{2})-(\ref{3}) and
\begin{equation}
c_{\delta}:=%
%TCIMACRO{\dint \limits_{\delta}^{\infty}}%
%BeginExpansion
{\displaystyle \int \limits_{\delta}^{\infty}}
%EndExpansion
\left(  1+\ln \frac{t}{r}\right)  \sup_{x\in{\mathbb{R}^{n}}}\varphi_{1}\left(
x,t\right)  \frac{dt}{t^{\frac{n}{q}+1}}<\infty \label{6-}%
\end{equation}
for every $\delta>0$, and
\begin{equation}
\int \limits_{r}^{\infty}\left(  1+\ln \frac{t}{r}\right)  \varphi_{1}\left(
x,t\right)  \frac{dt}{t^{\frac{n}{q}+1}}\leq C_{0}\frac{\varphi_{2}%
(x,r)}{r^{\frac{n}{q}}},\label{7-}%
\end{equation}
and for $q<s$ the pair $(\varphi_{1},\varphi_{2})$ satisfies conditions
(\ref{2})-(\ref{3}) and also%
\begin{equation}
c_{\delta^{\prime}}:=%
%TCIMACRO{\dint \limits_{\delta^{\prime}}^{\infty}}%
%BeginExpansion
{\displaystyle \int \limits_{\delta^{\prime}}^{\infty}}
%EndExpansion
\left(  1+\ln \frac{t}{r}\right)  \sup_{x\in{\mathbb{R}^{n}}}\varphi
_{1}(x,t)\frac{dt}{t^{\frac{n}{q}-\frac{n}{s}+1}}<\infty \label{8-}%
\end{equation}
for every $\delta^{\prime}>0$, and%
\begin{equation}
\int \limits_{r}^{\infty}\left(  1+\ln \frac{t}{r}\right)  \varphi_{1}%
(x,t)\frac{dt}{t^{\frac{n}{q}-\frac{n}{s}+1}}\leq C_{0}\frac{\varphi_{2}%
(x,r)}{r^{\frac{n}{q}-\frac{n}{s}}},\label{9-}%
\end{equation}
where $C_{0}$ does not depend on $x\in{\mathbb{R}^{n}}$ and $r>0$.

Then the operator $T_{\Omega,b,\alpha}$ is bounded from $VM_{p,\varphi_{1}}$
to $VM_{q,\varphi_{2}}$. Moreover,%
\[
\left \Vert T_{\Omega,b,\alpha}f\right \Vert _{VM_{q,\varphi_{2}}}%
\lesssim \left \Vert b\right \Vert _{\ast}\left \Vert f\right \Vert _{VM_{p,\varphi
_{1}}}.
\]

\end{theorem}

\begin{proof}
Similar to the proof of Theorem \ref{teo3}, let $s^{\prime}\leq p$. The
estimation of the norm of the operator follows from Lemma 4 in \cite{Gurbuz2}
and condition (\ref{7-})
\begin{align*}
\Vert T_{\Omega,b,\alpha}f\Vert_{VM_{q,\varphi_{2}}}  & =\sup \limits_{x\in
{\mathbb{R}^{n}},r>0}\varphi_{2}(x,r)^{-1}\Vert T_{\Omega,b,\alpha}%
f\Vert_{L_{q}(B(x,r))}\\
& \lesssim \left \Vert b\right \Vert _{\ast}\sup_{x\in{\mathbb{R}^{n},}%
r>0}\varphi_{2}\left(  x,r\right)  ^{-1}r^{\frac{n}{q}}\int \limits_{r}%
^{\infty}\left(  1+\ln \frac{t}{r}\right)  \left \Vert f\right \Vert
_{L_{p}\left(  B\left(  x,t\right)  \right)  }\frac{dt}{t^{\frac{n}{q}+1}}\\
& \lesssim \left \Vert b\right \Vert _{\ast}\sup_{x\in{\mathbb{R}^{n},}%
r>0}\varphi_{2}\left(  x,r\right)  ^{-1}r^{\frac{n}{q}}\int \limits_{r}%
^{\infty}\left(  1+\ln \frac{t}{r}\right)  \varphi_{1}\left(  x,t\right)
\left[  \varphi_{1}\left(  x,t\right)  ^{-1}\left \Vert f\right \Vert
_{L_{p}\left(  B\left(  x,t\right)  \right)  }\right]  \frac{dt}{t^{\frac
{n}{q}+1}}\\
& \lesssim \left \Vert b\right \Vert _{\ast}\left \Vert f\right \Vert
_{VM_{p,\varphi_{1}}}\sup_{x\in{\mathbb{R}^{n},}r>0}\varphi_{2}\left(
x,r\right)  ^{-1}r^{\frac{n}{q}}\int \limits_{r}^{\infty}\left(  1+\ln \frac
{t}{r}\right)  \varphi_{1}\left(  x,t\right)  \frac{dt}{t^{\frac{n}{q}+1}}\\
& \lesssim \left \Vert b\right \Vert _{\ast}\left \Vert f\right \Vert
_{VM_{p,\varphi_{1}}}.
\end{align*}
Thus we only have to prove that%
\begin{equation}
\lim \limits_{r\rightarrow0}\sup \limits_{x\in{\mathbb{R}^{n}}}\varphi
_{1}(x,r)^{-1}\, \left \Vert f\right \Vert _{L_{p}\left(  B\left(  x,r\right)
\right)  }=0\text{ implies }\lim \limits_{r\rightarrow0}\sup \limits_{x\in
{\mathbb{R}^{n}}}\varphi_{2}(x,r)^{-1}\, \left \Vert T_{\Omega,b,\alpha
}f\right \Vert _{L_{q}\left(  B\left(  x,r\right)  \right)  }=0.\label{12-}%
\end{equation}

To show that $\sup \limits_{x\in{\mathbb{R}^{n}}}\frac{\left \Vert
T_{\Omega,b,\alpha}f\right \Vert _{L_{q}\left(  B\left(  x,r\right)  \right)
}}{\varphi_{2}(x,r)}<\epsilon$ for small $r$, we split the right-hand side of
the first inequality in Lemma 4 in \cite{Gurbuz2}:%
\[
\frac{r^{-\frac{n}{q}}\left \Vert T_{\Omega,b,\alpha}f\right \Vert
_{L_{q}\left(  B\left(  x,r\right)  \right)  }}{\varphi_{2}(x,r)}\leq C\left[
I_{\delta_{0}}\left(  x,r\right)  +J_{\delta_{0}}\left(  x,r\right)  \right]
,
\]
where $\delta_{0}>0$ (we may take $\delta_{0}<1$), and
\[
I_{\delta_{0}}\left(  x,r\right)  :=\left \Vert b\right \Vert _{\ast}%
\frac{r^{\frac{n}{q}}}{\varphi_{2}(x,r)}%
%TCIMACRO{\dint \limits_{r}^{\delta_{0}}}%
%BeginExpansion
{\displaystyle \int \limits_{r}^{\delta_{0}}}
%EndExpansion
\left(  1+\ln \frac{t}{r}\right)  \varphi_{1}\left(  x,t\right)  t^{-\frac
{n}{q}-1}\left(  \varphi_{1}\left(  x,t\right)  ^{-1}\left \Vert f\right \Vert
_{L_{p}\left(  B\left(  x,t\right)  \right)  }\right)  dt,
\]
and%
\[
J_{\delta_{0}}\left(  x,r\right)  :=\left \Vert b\right \Vert _{\ast}%
\frac{r^{\frac{n}{q}}}{\varphi_{2}(x,r)}%
%TCIMACRO{\dint \limits_{\delta_{0}}^{\infty}}%
%BeginExpansion
{\displaystyle \int \limits_{\delta_{0}}^{\infty}}
%EndExpansion
\left(  1+\ln \frac{t}{r}\right)  \varphi_{1}\left(  x,t\right)  t^{-\frac
{n}{q}-1}\left(  \varphi_{1}\left(  x,t\right)  ^{-1}\left \Vert f\right \Vert
_{L_{p}\left(  B\left(  x,t\right)  \right)  }\right)  dt,
\]
and $r<\delta_{0}$ and and the rest of the proof is the same as the proof of
Theorem \ref{teo3}. Thus, we can prove that (\ref{12-}).

For the case of $q<s$, we can also use the same method, so we omit the
details, which completes the proof.
\end{proof}

\begin{remark}
Conditions (\ref{6-}) and (\ref{8-}) are not needed in the case when
$\varphi(x,r)$ does not depend on $x$, since (\ref{6-}) follows from
(\ref{7-}) and similarly, (\ref{8-}) follows from (\ref{9-}) in this case.
\end{remark}

\begin{corollary}
Under the conditions of Theorem \ref{teo4}, the operators $M_{\Omega,b,\alpha
}$ and $[b,\overline{T}_{\Omega,\alpha}]$ are bounded from $VM_{p,\varphi_{1}%
}$ to $VM_{q,\varphi_{2}}$.
\end{corollary}

\begin{corollary}
\label{Corollary1*}Let $\Omega \in L_{s}(S^{n-1})$, $1<s\leq \infty$, be
homogeneous of degree zero. Let $0<\alpha,\lambda<n$, $1<p<\frac{n-\lambda
}{\alpha}$, $\frac{1}{p}-\frac{1}{q}=\frac{\alpha}{n}$, and $\frac{\lambda}%
{p}=\frac{\mu}{q}$ and $b\in BMO\left(  {\mathbb{R}^{n}}\right)  $. Let
$T_{\Omega,b,\alpha}$ be a sublinear operator satisfying condition
(\ref{1**}). Then for $s^{\prime}\leq p$ or $q<s$, we have%
\[
\left \Vert T_{\Omega,b,\alpha}f\right \Vert _{VM_{q,\mu}}\lesssim \left \Vert
b\right \Vert _{\ast}\left \Vert f\right \Vert _{VM_{p,\lambda}}.
\]

\end{corollary}

\begin{proof}
Similar to the proof of Corollary \ref{Corollary0*}, Let $s^{\prime}\leq p$.
By using $\varphi_{1}\left(  x,r\right)  =r^{\frac{\lambda}{p}}$ and
$\varphi_{2}\left(  x,r\right)  =r^{\frac{\mu}{q}}$ in the proof of Theorem
\ref{teo4} and condition (\ref{7-}), it follows that%
\begin{align*}
\Vert T_{\Omega,b,\alpha}f\Vert_{VM_{q,\mu}}  & \lesssim \left \Vert
b\right \Vert _{\ast}\sup_{x\in{\mathbb{R}^{n},}r>0}r^{-\frac{\mu}{q}}%
r^{\frac{n}{q}}\int \limits_{r}^{\infty}\left(  1+\ln \frac{t}{r}\right)
\varphi_{1}\left(  x,t\right)  t^{-\frac{n}{q}-1}\left(  \varphi_{1}\left(
x,t\right)  ^{-1}\left \Vert f\right \Vert _{L_{p}\left(  B\left(  x,t\right)
\right)  }\right)  dt\\
& \lesssim \left \Vert b\right \Vert _{\ast}\left \Vert f\right \Vert
_{VM_{p,\lambda}}\sup_{x\in{\mathbb{R}^{n},}r>0}r^{\frac{n-\mu}{q}}%
\int \limits_{r}^{\infty}\left(  1+\ln \frac{t}{r}\right)  r^{\frac{\lambda}{p}%
}\frac{dt}{t^{\frac{n}{q}+1}}\\
& \lesssim \left \Vert b\right \Vert _{\ast}\left \Vert f\right \Vert
_{VM_{p,\lambda}},
\end{align*}
for the case of $q<s$, we can also use the same method, so we omit the
details, which completes the proof.
\end{proof}

\begin{corollary}
Under the conditions of Corollary \ref{Corollary1*}, the operators
$M_{\Omega,b,\alpha}$ and $[b,\overline{T}_{\Omega,\alpha}]$ are bounded from
$VM_{p,\lambda}$ to $VM_{q,\mu}$.
\end{corollary}

\textbf{Acknowledgement: }The part entitled "Sublinear operators with rough
kernel generated by Calder\'{o}n-Zygmund operators and their commutators on
vanishing generalized Morrey spaces"of this study has been given as the
plenary talk by the author at the \textquotedblleft International Conference
on Mathematics and Mathematics Education (ICMME-2017)\textquotedblright \ May
11-13, 2017 in Sanliurfa, Turkey.

\end{document}